\documentclass[11pt]{article}

\usepackage{amsmath,amssymb,amsthm,amsfonts,mathrsfs}
\usepackage{geometry}
\usepackage{bm}
\usepackage{color}
\usepackage{graphicx}

\numberwithin{equation}{section}

\newtheorem{theorem}{Theorem}[section]

\newtheorem{proposition}{Proposition}[section]
\newtheorem{Corollary}{Corollary}[section]
\newtheorem{remark}{Remark}[section]

\title{Spiking Neural Networks with Elephant Reinforcement}

\author{
Fernando A. Najman \thanks{
 f.najman@ufabc.edu.br, fanajman@gmail.com \\
UFABC -- Centro de Matem\'atica, Computa\c{c}\~ao e Cogni\c{c}\~ao\\
Avenida dos Estados, 5001, 
Bangu, Santo Andr\'e -- SP, Brazil
}
\and
Ioannis Papageorgiou\thanks{
i.papageorgiou@ufabc.edu.br, 
papyannis@yahoo.com\\
UFABC -- Centro de Matem\'atica, Computa\c{c}\~ao e Cogni\c{c}\~ao\\
Avenida dos Estados, 5001, 
Bangu, Santo Andr\'e -- SP, Brazil
}
\and
Sabricia K. Cauanny A. da Silveira\thanks{
cauanny.silveira@aluno.ufabc.edu.br, 
cauannysilveira45@gmail.com\\
UFABC -- Centro de Matem\'atica, Computa\c{c}\~ao e Cogni\c{c}\~ao\\
Avenida dos Estados, 5001, 
Bangu, Santo Andr\'e -- SP, Brazil
}
}

\date{}

\begin{document}

\maketitle

\begin{abstract}
We introduce a finite stochastic spiking-neuron network with Elephant-type memory, in which past firing activity modifies future excitability through a reinforcement-dependent threshold. For a bounded hard-threshold firing rate, we prove non-explosion of the finite system and obtain conditional exponential contraction in (1)-Wasserstein distance on a truncated potential space. We then formulate the corresponding replica mean-field dynamics and establish global existence, uniqueness in law, and non-explosion of the nonlinear process, together with a characterization of its invariant measures. Numerical experiments show that Elephant memory produces a (p)-dependent decline in firing activity, alters extinction behaviour, and yields finite-network dynamics closely matched by the replica mean-field approximation.
\end{abstract}

\maketitle

\section{Introduction}

Neuronal spike generation depends not only on the instantaneous state of a
neuron but also on its previous activity and on the activity of the network in
which it is embedded. At the level of individual neurons, spike-frequency
adaptation is a prominent manifestation of this history dependence: under
sustained stimulation, firing rates commonly decrease as intrinsic mechanisms
progressively modify neuronal excitability
\cite{BendaHerz2003,LaCameraEtAl2006}. Experimental studies have also shown
that the voltage threshold for spike initiation can vary with the preceding
membrane-potential and firing history \cite{FontainePenaBrette2014}. In
neocortical pyramidal neurons, spike-triggered threshold and current adaptation
may persist over several timescales and can display approximately power-law
relaxation \cite{PozzoriniEtAl2013}. These observations motivate stochastic
models in which past spike activity modifies future firing probabilities.

A broad mathematical literature has developed stochastic models of interacting
spiking neurons. Classical approaches include stochastic integrate-and-fire
systems and interacting point-process models
\cite{Stein,Tuc,Cot92,Br-Ha99,Br2000}. Early stochastic approaches to biologically plausible neural networks,
including systems with excitatory and inhibitory interactions, were developed
by Turova \cite{Turova1997Hourglass,Turova1997Stochastic}, providing part of
the probabilistic background for later interacting spiking-neuron models. The Galves--L\"ocherbach framework
introduced an infinite system of interacting chains with variable-length
memory as a stochastic model for biological neural networks
\cite{GalvesLocherbach2013}. Subsequent work established, among other
properties, phase transitions, metastability, stability and mean-field
behaviour for related neuronal systems
\cite{F-G-L,L-M,C17,C-D-L-O,D-L,D-L-O}. Numerical aspects of extinction in
spiking-neuron networks were investigated by 
\cite{romaro}, providing a useful computational perspective on the lifetime of
activity in finite neuronal systems.

A complementary line of research concerns quantitative properties of
degenerate jump processes and interacting neuronal networks. Infinite systems of interacting spiking neurons under interaction assumptions going beyond uniform summability were studied in \cite{Papageorgiou2023}. Coercive  and concentration inequalities for related
degenerate jump processes were studied in 
\cite{Papageorgiou2020, HodaraPapageorgiou2019, HodaraPapageorgiou2022}. Replica mean-field methods were subsequently
developed for neural networks with both excitatory and inhibitory activity
\cite{Papageorgiou2025}, in connection with the replica mean-field framework
introduced by Baccelli and Taillefumier
\cite{B-T,BaTa20}.

The second ingredient of the present model comes from reinforcement and,
specifically, from the Elephant random walk. The Elephant random walk was
introduced by Sch\"utz and Trimper as a non-Markovian random walk whose next
increment depends on its accumulated past
\cite{SchutzTrimper2004}. Its asymptotic behaviour has since been studied by
martingale, urn and invariance-principle techniques
\cite{BaurBertoin2016,Bercu2017,Coletti2017a,Coletti2017b,
GutStadtmuellerVariations,GutStadtmuellerReview,Laulin2022martingale}.
An  asymptotic analysis of the Elephant random
walk was obtained in \cite{ColettiPapageorgiou2019} , while more  
recently \cite{ColettiGrisiPapageorgiou2026} established a Poincar\'e inequality and exponential decay for a
bounded Elephant random walk. These works
show that Elephant-type mechanisms provide a natural probabilistic framework
for modelling persistent or anti-persistent memory.

The present article is based on the model introduced  in
\cite{PapageorgiouElephant2026}. That work combines the
Galves--L\"ocherbach interacting-neuron framework with an Elephant-type
reinforcement mechanism, allowing previous neuronal activity to modify the
subsequent dynamics through an additional memory variable. The principal idea
is to enlarge the neuronal state by a signed reinforcement variable whose
successive increments obey an Elephant rule. In this way, a mechanism
originating in reinforced random walks is incorporated directly into a
stochastic spiking-neuron network.

Here we study a finite, bounded-rate version of this construction that is
particularly suitable for both rigorous analysis and numerical investigation.
Each neuron carries a membrane potential \(V_i\), a signed reinforcement
variable \(S_i\), and a reinforcement counter \(K_i\). Whenever neuron \(i\)
fires, the pair \((S_i,K_i)\) is updated according to an Elephant-type rule.
The positive part \(S_i^+\) modifies the firing threshold through
$
V_i>\alpha S_i^+.
$
Thus the accumulated firing history feeds back into future excitability. For
\(p>1/2\), the reinforcement mechanism favours persistence of previous
reinforcement directions, whereas for \(p<1/2\) it favours
anti-persistence. The case \(p=1/2\) provides the natural memory-neutral
reference.

The bounded hard-threshold firing rate used in the present paper also makes
the finite system particularly transparent mathematically. We first prove
non-explosion by obtaining a uniform bound on the total event intensity.
We then study stability of the membrane-potential component. On a truncated
potential state space, and conditionally on a common reinforcement history, we
derive exponential contraction in the \(1\)-Wasserstein distance under an
explicit dissipativity condition. This identifies a regime in which leakage
dominates the expansion caused by asynchronous firing events.

We next introduce a randomly routed replica system and formulate the
corresponding nonlinear replica mean-field dynamics. Following the general RMF
philosophy of Baccelli and Taillefumier
\cite{B-T,BaTa20} and the neuronal RMF developments in
\cite{Papageorgiou2025}, we derive the nonlinear generator,
self-consistency relations, Kolmogorov forward equations and exact moment
identities. We prove global existence, uniqueness in law and non-explosion of
the nonlinear RMF process through a fixed-point argument. Spatial homogeneity
is shown to be preserved when the initial laws are translation invariant.
We further characterize the invariant probability measures of the full
reinforced nonlinear process and prove that every invariant probability
measure is supported on silent states.

The numerical part is a central component of the present work. It  complements
the extinction-time simulations of 
\cite{romaro}. Under sustained homogeneous input, we investigate whether the
Elephant reinforcement mechanism can generate adaptation-like firing
dynamics. We compare the reinforced network with a no-memory control, examine
the dependence on the memory parameter \(p\), and follow simultaneously the
firing activity and the effective reinforcement-dependent threshold. We also
compare power-law and single-exponential descriptions of the resulting
adaptation curves. In the undriven network, we study the effect of
reinforcement on extinction times and total firing activity, and we compare
the finite-network dynamics with the corresponding replica mean-field
approximation.

The aim of the present paper is therefore not to identify the reinforcement
variable with a specific ionic or molecular mechanism. Rather, it provides a
parsimonious stochastic description of long-range spike-history dependence
and investigates mathematically and numerically how such memory can modify
network excitability. The combination of Elephant reinforcement,
Galves--L\"ocherbach-type neuronal interactions and replica mean-field
dynamics provides a framework in which memory, stability, extinction and
adaptation-like behaviour can be studied within a common stochastic model.

\section{The reinforced finite network}

We consider a finite network of \(N\) interacting neurons arranged on a
discrete ring. The neurons are indexed by \(i\in\{1,\ldots,N\}\), with indices
understood modulo \(N\), so that the two neighbours of neuron \(i\) are
\(i-1\) and \(i+1\). The state of the network at time \(t\) is
\[
Y_t=(V_t,S_t,K_t)
\in
\mathbb N^N\times\mathbb Z^N\times\mathbb N^N.
\]
For each neuron \(i\), the coordinate \(V_i(t)\) denotes its membrane
potential, \(S_i(t)\) is its signed reinforcement variable, and \(K_i(t)\)
records the number of reinforcement updates that have occurred up to time
\(t\).

The membrane potential changes through firing, synaptic input, and leakage.
When neuron \(i\) fires, its potential is reset to zero and each of its two
neighbours receives one unit of potential. At the same event, the internal
reinforcement state of neuron \(i\) is updated according to
\[
S_i\longmapsto S_i+\xi_i,
\qquad
K_i\longmapsto K_i+1.
\]
The reinforcement variable affects subsequent firing through the threshold $\alpha S_i^+$, $S_i^+=\max\{S_i,0\}$. Thus positive reinforcement raises the firing threshold, while the membrane
potential records the excitation received from the rest of the network.

Conditionally on the state immediately before the firing,
\[
\mathbb P(\xi_i=+1\mid S_i=s,K_i=k)
=
q^+(s,k)
=
\frac12+\frac{2p-1}{2}\frac{s}{k},
\]
and
\[
\mathbb P(\xi_i=-1\mid S_i=s,K_i=k)
=
q^-(s,k)
=
\frac12-\frac{2p-1}{2}\frac{s}{k}.
\]

Before studying the stability and mean-field behaviour of the model, we must
verify that these transitions define a process for all finite times. In
particular, it is necessary to exclude the accumulation of infinitely many
firing and leakage events in a bounded time interval. For the bounded
hard-threshold firing rate, this follows from a uniform bound on the total
jump intensity.

Throughout the paper, we assume
$N\ge 3$, $\alpha\ge 0$,  $\gamma>0$, $p\in[0,1]$.

\subsection{Finite-volume model and non-explosion}

\begin{theorem}
\label{thm:finite-nonexplosion}
Let \(N<\infty\). Consider the process
\[
Y_t=(V_t,S_t,K_t)
\in \mathbb N^N\times\mathbb Z^N\times\mathbb N^N.
\]
Assume that site \(i\) fires with rate
\[
\lambda_i(V,S,K)
=
\mathbf 1_{\{V_i>\alpha S_i^+\}},
\qquad
S_i^+=\max\{S_i,0\},
\]
and leaks with rate
\[
\gamma\mathbf 1_{\{V_i>0\}},
\qquad \gamma>0.
\]
When site \(i\) fires, \(V_i\) is reset to \(0\), its two neighbours receive
one unit of potential, and
\[
S_i\mapsto S_i+\xi_i,
\qquad
K_i\mapsto K_i+1.
\]
Then the process is non-explosive. More precisely, if \(N_t\) denotes the
number of events up to time \(t\), then
\[
\mathbb E[N_t]\leq (\gamma+1)Nt<\infty.
\]
\end{theorem}

\begin{proof}
The total event rate is
\[
\Lambda(V,S,K)
=
\gamma\sum_{i=1}^N\mathbf 1_{\{V_i>0\}}
+
\sum_{i=1}^N\mathbf 1_{\{V_i>\alpha S_i^+\}}.
\]
Since
\[
\sum_{i=1}^N\mathbf 1_{\{V_i>0\}}\leq N
\]
and
\[
\sum_{i=1}^N\mathbf 1_{\{V_i>\alpha S_i^+\}}\leq N,
\]
we get
\[
\Lambda(V,S,K)\leq \gamma N+N=(\gamma+1)N.
\]
Therefore the total number of events \(N_t\) is stochastically dominated by a
Poisson process with rate \((\gamma+1)N\). Hence
\[
\mathbb E[N_t]\leq (\gamma+1)Nt<\infty.
\]
Thus the process is non-explosive.
\end{proof}

The bounded firing rate prevents the reinforcement mechanism from producing
an accumulation of events in finite time. Although the reinforcement state
retains information from successive firings, the instantaneous firing rate of
each neuron remains bounded by one. The finite reinforced network is therefore
well defined on the whole time interval \([0,\infty)\), which allows us to
study its stability and mean-field dynamics.

\subsection{Conditional Wasserstein contraction}

We next examine
the stability of its membrane-potential dynamics. The reinforcement variables
introduce a time-dependent firing threshold, so we compare two potential
processes conditionally on the same reinforcement history. We also truncate
the potential at a finite level \(M\), which makes it possible to control the
largest increase in distance caused by a firing occurring in only one of the
two processes. The following result identifies a regime in which the
contractive effect of leakage dominates this possible expansion.
\begin{theorem}\label{thm:frozen-contraction}
$V_t\in\{0,1,\ldots,M\}^N$. 
Assume a common reinforcement history \(S(t)\). The firing rate is
\[
\lambda_i(t,V)
=
\mathbf 1_{\{V_i>\alpha S_i^+(t)\}}.
\]
Let
$d(V,\widehat V)=\sum_{i=1}^N |V_i-\widehat V_i|$. 
Assume
\[
\sum_{i=1}^N
|\lambda_i(t,V)-\lambda_i(t,\widehat V)|
\leq
L d(V,\widehat V).
\]
If
$\gamma>(M+2)L,
$
then, with
$
\delta=\gamma-(M+2)L>0,
$ 
we have
\[
W_1(\mu P_t,\nu P_t)
\leq
e^{-\delta t}W_1(\mu,\nu).
\]
\end{theorem}

\begin{proof}
Couple leakages with the same clocks. If site \(i\) leaks in both copies,
$
V_i\mapsto0$, $\widehat V_i\mapsto0$. 
Hence
\[
|0-0|-|V_i-\widehat V_i|
=
-|V_i-\widehat V_i|.
\]
Therefore
\[
I_{\mathrm{leak}}
=
-\gamma\sum_{i=1}^N |V_i-\widehat V_i|
=
-\gamma d(V,\widehat V).
\]

For firing, write
$
\lambda_i=\lambda_i(t,V)$, $\widehat\lambda_i=\lambda_i(t,\widehat V)$. 
Both copies fire at site \(i\) with rate
$
\min\{\lambda_i,\widehat\lambda_i\},
$ 
only \(V\) fires with rate
$
(\lambda_i-\widehat\lambda_i)^+,
$
and only \(\widehat V\) fires with rate
$
(\widehat\lambda_i-\lambda_i)^+.
$ Synchronous firings are non-expansive. Indeed, the same map is applied to both:
\[
V_i\mapsto0,\qquad
V_{i\pm1}\mapsto \min\{V_{i\pm1}+1,M\}.
\]
Since \(x\mapsto\min\{x+1,M\}\) is \(1\)-Lipschitz,
\[
d(T_iV,T_i\widehat V)\leq d(V,\widehat V).
\]

If only one copy fires at site \(i\), then
\[
|0-\widehat V_i|-|V_i-\widehat V_i|\leq M,
\]
and for the two neighbours,
\[
|\min\{V_{i\pm1}+1,M\}-\widehat V_{i\pm1}|
-
|V_{i\pm1}-\widehat V_{i\pm1}|
\leq1.
\]
Thus one asynchronous firing increases the distance by at most
$
M+2.
$
Therefore
\[
I_{\mathrm{fire}}
\leq
(M+2)
\sum_{i=1}^N
|\lambda_i-\widehat\lambda_i|.
\]
Using the Lipschitz assumption,
\[
I_{\mathrm{fire}}
\leq
(M+2)Ld(V,\widehat V).
\]

Hence
\[
\mathcal L_c d(V,\widehat V)
\leq
-\gamma d(V,\widehat V)
+
(M+2)Ld(V,\widehat V).
\]
So
\[
\mathcal L_c d(V,\widehat V)
\leq
-\delta d(V,\widehat V),
\]
where
$
\delta=\gamma-(M+2)L>0.
$

By Dynkin's formula and Gronwall's inequality,
\[
\mathbb E[d(V_t,\widehat V_t)]
\leq
e^{-\delta t}d(V_0,\widehat V_0).
\]
Taking the infimum over all couplings gives
\[
W_1(\mu P_t,\nu P_t)
\leq
e^{-\delta t}W_1(\mu,\nu).
\]
\end{proof}
The theorem gives a quantitative stability criterion for the potential
dynamics under a fixed reinforcement history. The condition
$
\gamma>(M+2)L
$
expresses the balance between the contraction produced by leakage and the
expansion that may result from asynchronous firing events. When leakage is
sufficiently strong, the expected distance between two copies decreases
exponentially at rate
$
\delta=\gamma-(M+2)L.
$
Thus, even though the firing thresholds retain information from previous
spikes, the membrane-potential dynamics remain contractive in this
dissipative regime.
\section{Replica formulation}

We consider a network with \(N\) neuron types, indexed by
$
I=\{1,\ldots,N\},
$
with periodic boundary conditions. Thus, the postsynaptic neighbours of neuron
type \(i\) are
\[
\mathcal N_i=\{i-1,i+1\},
\]
where the indices are understood modulo \(N\).

The local state of a neuron is
\[
X=(V,S,K)\in\mathbb N\times\mathbb Z\times\mathbb N^{*},
\]
where \(V\) is the membrane potential, \(S\) is the reinforcement variable, and
\(K\) is the number of previous reinforcement updates.

For \(s\in\mathbb Z\), write
$
s^{+}=\max\{s,0\}.
$
The firing intensity is
$
\lambda(v,s,k)
=
\mathbf 1_{\{v>\alpha s^{+}\}}.
$
More generally, one could take
$
\lambda(v,s,k)
=
\phi(v)\mathbf 1_{\{v>\alpha s^{+}\}},
$
but throughout this section we consider the bounded-rate case
$
\phi(v)=1.
$

At each firing, the reinforcement increment \(\xi\in\{-1,+1\}\) has conditional
distribution
\[
q^{+}(s,k)
:=
\mathbb P(\xi=+1\mid S=s,K=k)
=
\frac12
+
\frac{2p-1}{2}\frac{s}{k},
\]
and
\[
q^{-}(s,k)
:=
\mathbb P(\xi=-1\mid S=s,K=k)
=
\frac12
-
\frac{2p-1}{2}\frac{s}{k}.
\]
We assume that
$
|s|\leq k,
$
so that \(q^{+}(s,k)\) and \(q^{-}(s,k)\) are well-defined probabilities.

\subsection{The system with $M$ replicas}

For every neuron type $i \in I$, introduce $M$ replicas indexed by
$
m\in\{1,\ldots,M\}.
$
The state of replica \(m\) of neuron type \(i\) is denoted by
\[
X_i^{m,M}(t)
=
\left(
V_i^{m,M}(t),
S_i^{m,M}(t),
K_i^{m,M}(t)
\right).
\]

The complete state is therefore
\[
X^{M}(t)
=
\left(
X_i^{m,M}(t)
\right)_{
  1\leq i\leq N,\,
  1\leq m\leq M
}.
\]

Each neuron \((i,m)\) is affected by three kinds of events.

\paragraph{Leakage.}
If
$
V_i^{m,M}>0,
$
the neuron leaks at rate \(\gamma\), and
\[
\left(
V_i^{m,M},
S_i^{m,M},
K_i^{m,M}
\right)
\longmapsto
\left(
0,
S_i^{m,M},
K_i^{m,M}
\right).
\]

\paragraph{Firing and reinforcement.}
Neuron \((i,m)\) fires at rate
\[
\lambda_i^{m,M}
=
\mathbf 1_{\{
V_i^{m,M}>
\alpha(S_i^{m,M})^{+}
\}}.
\]
At a firing time, its local state changes according to
$
V_i^{m,M}\longmapsto0,
$
$
S_i^{m,M}\longmapsto S_i^{m,M}+\xi_i^{m},
$
and
$
K_i^{m,M}\longmapsto K_i^{m,M}+1,
$
where, conditionally on the state immediately before the firing,
\[
\mathbb P\left(
\xi_i^{m}=+1
\,\middle|\,
S_i^{m,M},K_i^{m,M}
\right)
=
q^{+}\left(
S_i^{m,M},K_i^{m,M}
\right).
\]

\paragraph{Random replica routing.}
When neuron \((i,m)\) fires, it sends one unit of potential to one uniformly
chosen replica of each postsynaptic neuron type.

For every \(j\in\mathcal N_i\), let
$
R_{i\to j}^{m}
$
be uniformly distributed on
$
\{1,\ldots,M\},
$
independently for different postsynaptic types and different firing events.

The routed interaction is
\[
V_j^{R_{i\to j}^{m},M}
\longmapsto
V_j^{R_{i\to j}^{m},M}+1,
\qquad
j\in\mathcal N_i.
\]

Thus a firing of replica \((i,m)\) does not necessarily excite the neurons
\((i-1,m)\) and \((i+1,m)\) in the same replica. Instead, it excites uniformly
chosen replicas of the two postsynaptic neuron types.

\subsection{Generator of the replica system}

Let \(x=(v,s,k)\) denote a configuration of the complete replicated system.
For a bounded cylinder function \(F\), the generator is
\[
\mathcal L^{M}F
=
\mathcal L_{\mathrm{leak}}^{M}F
+
\mathcal L_{\mathrm{fire}}^{M}F.
\]

The leakage part is
\[
\mathcal L_{\mathrm{leak}}^{M}F(x)
=
\gamma
\sum_{i=1}^{N}
\sum_{m=1}^{M}
\mathbf 1_{\{v_i^m>0\}}
\left[
F\left(\ell_i^m x\right)-F(x)
\right],
\]
where \(\ell_i^m x\) is obtained by replacing \(v_i^m\) by \(0\).

For \(\xi\in\{-1,+1\}\) and
$
r=(r_{-},r_{+})\in\{1,\ldots,M\}^{2},
$
let
$
\Delta_{i,m}^{\xi,r}x
$
be the configuration obtained by
\[
v_i^m\longmapsto0,
\qquad
s_i^m\longmapsto s_i^m+\xi,
\qquad
k_i^m\longmapsto k_i^m+1,
\]
and
\[
v_{i-1}^{r_{-}}\longmapsto v_{i-1}^{r_{-}}+1,
\qquad
v_{i+1}^{r_{+}}\longmapsto v_{i+1}^{r_{+}}+1.
\]

Then
\begin{align*}
\mathcal L_{\mathrm{fire}}^{M}F(x)
=
\sum_{i=1}^{N}
\sum_{m=1}^{M}
\lambda(v_i^m,s_i^m,k_i^m)
\sum_{\xi\in\{-1,+1\}}
q^{\xi}(s_i^m,k_i^m)
\frac{1}{M^{2}}
\sum_{r_{-}=1}^{M}
\sum_{r_{+}=1}^{M}
\left[
F\left(\Delta_{i,m}^{\xi,r}x\right)-F(x)
\right].
\end{align*}

Here
$
q^{+1}=q^{+},
$ $
q^{-1}=q^{-}.
$

\section{The nonlinear RMF process}

Let
$
X_i(t)=\bigl(V_i(t),S_i(t),K_i(t)\bigr)
$ 
denote the state of a representative neuron of type \(i\).

Define its mean firing intensity by
\[
\beta_i(t)
=
\mathbb E\left[
\lambda\bigl(V_i(t),S_i(t),K_i(t)\bigr)
\right].
\]
For the bounded hard-threshold model,
\[
\boxed{
\beta_i(t)
=
\mathbb P\left(
V_i(t)>\alpha S_i^{+}(t)
\right).
}
\]

In the nonlinear RMF formulation, the representative neuron of type \(i\) receives two
independent Poisson input processes:
\[
A_{i-1\to i}
\quad\text{and}\quad
A_{i+1\to i},
\]
with deterministic time-dependent intensities
\[
\beta_{i-1}(t)
\quad\text{and}\quad
\beta_{i+1}(t),
\]
respectively.

Hence the total input intensity of neuron \(i\) is
\[
a_i(t)
=
\beta_{i-1}(t)+\beta_{i+1}(t).
\]

At an input time,
$
V_i\longmapsto V_i+1,
$
while \(S_i\) and \(K_i\) remain unchanged.

The representative neuron also leaks at rate
$
\gamma\mathbf 1_{\{V_i>0\}}$,
and fires at rate
$
\lambda(V_i,S_i,K_i).
$

Therefore the resulting process is nonlinear in the sense of McKean:
its transition rates depend on its own law through
\[
\beta_i(t)
=
\int\lambda(x)\,\mu_i(t)(dx).
\]

\subsection{Nonlinear generator}

Let
$
f:\mathbb N\times\mathbb Z\times\mathbb N^{*}
\longrightarrow\mathbb R
$
be bounded. For a prescribed family of measures
$
\boldsymbol{\mu}
=
(\mu_1,\ldots,\mu_N),
$
define
\[
\beta_i(\mu_i)
=
\int
\lambda(v,s,k)\,
\mu_i(dv,ds,dk).
\]

The nonlinear generator of the representative neuron of type \(i\) is
\begin{align}
\mathcal A_{i,\boldsymbol{\mu}}f(v,s,k)
={}&
\gamma\mathbf 1_{\{v>0\}}
\left[
f(0,s,k)-f(v,s,k)
\right]
\label{eq:rmf-generator}
\\
&+
\lambda(v,s,k)
\sum_{\xi\in\{-1,+1\}}
q^{\xi}(s,k)
\left[
f(0,s+\xi,k+1)-f(v,s,k)
\right]
\nonumber\\
&+
\left[
\beta_{i-1}(\mu_{i-1})
+
\beta_{i+1}(\mu_{i+1})
\right]
\left[
f(v+1,s,k)-f(v,s,k)
\right].
\nonumber
\end{align}

The law \(\mu_i(t)\) must satisfy
\[
\frac{d}{dt}
\int f\,d\mu_i(t)
=
\int
\mathcal A_{i,\boldsymbol{\mu}(t)}f
\,d\mu_i(t)
\]
for every bounded test function \(f\).

Equivalently,
\[
\partial_t\mu_i(t)
=
\mathcal A_{i,\boldsymbol{\mu}(t)}^{*}\mu_i(t).
\]

Together with
\[
\beta_i(t)
=
\int\lambda\,d\mu_i(t),
\]
this is the RMF self-consistency system.

\subsection{Coordinate form of the forward equations}

For every neuron type \(i\), write
\[
p_i(t;v,s,k)
=
\mathbb P\bigl(V_i(t)=v,\ S_i(t)=s,\ K_i(t)=k\bigr),
\]
for \((v,s,k)\in E\), and set
\[
a_i(t)=\beta_{i-1}(t)+\beta_{i+1}(t).
\]
We use the convention that
$
p_i(t;v,s,k)=0
$
whenever \((v,s,k)\notin E\). In particular, terms with \(k-1<1\) vanish.

\begin{proposition}
\label{prop:forward-equations}
Let \(\boldsymbol\mu(t)\) be a solution of the nonlinear RMF equation. Then, for
all \((v,s,k)\in E\), the coordinate probabilities are absolutely continuous
in time and satisfy the following equations for almost every \(t\geq0\).

For \(v\geq1\),
\begin{align}
\frac{d}{dt}p_i(t;v,s,k)
=
a_i(t)p_i(t;v-1,s,k)
\label{eq:forward-positive}
-
\bigl[a_i(t)+\gamma+\lambda(v,s,k)\bigr]
 p_i(t;v,s,k).
\end{align}

For \(v=0\),
\begin{align}
\frac{d}{dt}p_i(t;0,s,k)
=&
-a_i(t)p_i(t;0,s,k)
\label{eq:forward-zero}
+
\gamma\sum_{u\geq1}p_i(t;u,s,k)
\\
&+
q^{+}(s-1,k-1)
\sum_{u\geq0}
\lambda(u,s-1,k-1)
 p_i(t;u,s-1,k-1)
\nonumber\\
&+
q^{-}(s+1,k-1)
\sum_{u\geq0}
\lambda(u,s+1,k-1)
 p_i(t;u,s+1,k-1).
\nonumber
\end{align}
Moreover,
\begin{equation}
\label{eq:forward-self-consistency}
\beta_i(t)
=
\sum_{(v,s,k)\in E}
\lambda(v,s,k)p_i(t;v,s,k).
\end{equation}
For the hard-threshold intensity,
\begin{equation}
\label{eq:forward-self-consistency-threshold}
\beta_i(t)
=
\sum_{s,k}
\sum_{v>\alpha s^+}
p_i(t;v,s,k).
\end{equation}
\end{proposition}

\begin{proof}
Fix \((v,s,k)\in E\) and define the bounded indicator test function
\[
f_{v,s,k}(u,r,n)
:=
\mathbf 1_{\{(u,r,n)=(v,s,k)\}}.
\]
The weak nonlinear equation gives
\begin{equation}
\label{eq:indicator-weak-forward}
p_i(t;v,s,k)
=
p_i(0;v,s,k)
+
\int_0^t
\int_E
\mathcal A_{i,\boldsymbol\mu(\tau)}f_{v,s,k}(x)
\,\mu_i(\tau)(dx)\,d\tau.
\end{equation}
The integrand is bounded in absolute value by \(2(\gamma+3)\). Hence the
right-hand side is absolutely continuous, and differentiation is legitimate for
almost every \(t\).

We identify all transitions entering and leaving the state \((v,s,k)\).

\medskip
\noindent
\textbf{Case \(v\geq1\).}

An incoming mean-field spike changes only the potential according to
$
(u,s,k)\longmapsto(u+1,s,k).
$
Therefore the only transition entering \((v,s,k)\) is
$
(v-1,s,k)\longmapsto(v,s,k),
$
which occurs at rate \(a_i(t)\). This gives the gain term
$a_i(t)p_i(t;v-1,s,k).$ 
There are three ways to leave \((v,s,k)\):
\begin{enumerate}
\item an incoming spike, at rate \(a_i(t)\), sends \(v\) to \(v+1\);
\item leakage, at rate \(\gamma\), sends \(v\) to zero because \(v>0\);
\item firing, at rate \(\lambda(v,s,k)\), sends the state to
\((0,s+\xi,k+1)\).
\end{enumerate}
Thus the total loss rate is
$
a_i(t)+\gamma+\lambda(v,s,k),
$
which proves \eqref{eq:forward-positive}.

\medskip
\noindent
\textbf{Case \(v=0\).}

An incoming spike leaves the state \((0,s,k)\) at rate \(a_i(t)\), producing the
loss term
$
-a_i(t)p_i(t;0,s,k).
$
There is no leakage loss because leakage acts only when the potential is
positive. Also, a firing event from \((0,s,k)\) cannot occur in the present
hard-threshold model, because
$
\lambda(0,s,k)
=
\mathbf 1_{\{0>\alpha s^+\}}
=0.
$

Every state \((u,s,k)\), with \(u\geq1\), reaches \((0,s,k)\) through leakage at
rate \(\gamma\). Summing over \(u\geq1\) yields
$
\gamma\sum_{u\geq1}p_i(t;u,s,k).
$

We next determine the firing transitions entering \((0,s,k)\). If the
reinforcement increment equals \(+1\), then the predecessor reinforcement state
must be \((s-1,k-1)\). Therefore the corresponding gain is
\[
q^+(s-1,k-1)
\sum_{u\geq0}
\lambda(u,s-1,k-1)
 p_i(t;u,s-1,k-1).
\]
If the increment equals \(-1\), then the predecessor must be
\((s+1,k-1)\), which gives
\[
q^-(s+1,k-1)
\sum_{u\geq0}
\lambda(u,s+1,k-1)
 p_i(t;u,s+1,k-1).
\]
Adding the loss and gain terms proves \eqref{eq:forward-zero}.

Finally, \eqref{eq:forward-self-consistency} is exactly the definition
\[
\beta_i(t)=\int_E\lambda(x)\,\mu_i(t)(dx).
\]
Substituting
$
\lambda(v,s,k)=\mathbf 1_{\{v>\alpha s^+\}}
$
gives \eqref{eq:forward-self-consistency-threshold}.
\end{proof}

\begin{remark}
Equations \eqref{eq:forward-positive}--\eqref{eq:forward-zero} conserve total
mass. Indeed, summing them over \((v,s,k)\in E\), all incoming-spike, leakage,
and firing terms cancel pairwise, and one obtains
\[
\frac{d}{dt}\sum_{(v,s,k)\in E}p_i(t;v,s,k)=0.
\]
Hence a probability initial condition remains a probability distribution.
\end{remark}

\subsection{Exact moment equations}

The coordinate functions \(v\), \(s\), and \(k\) are unbounded. We therefore
first justify their use through bounded truncations.

\begin{proposition}[Finite first moments and exact moment identities]
\label{prop:moment-identities}
Assume
\[
\mathbb E[V_i(0)+K_i(0)+|S_i(0)|]<\infty,
\qquad i=1,\ldots,N.
\]
Then for every \(T>0\),
\[
\sup_{0\leq t\leq T}
\mathbb E[V_i(t)+K_i(t)+|S_i(t)|]<\infty.
\]
Moreover, for almost every \(t\geq0\),
\begin{align}
\frac{d}{dt}\mathbb E[V_i(t)]
=
\beta_{i-1}(t)+\beta_{i+1}(t)
-
\gamma\mathbb E[V_i(t)]
\label{eq:moment-V}
-
\mathbb E\bigl[V_i(t)\lambda_i(t)\bigr],
\end{align}
\begin{equation}
\label{eq:moment-K}
\frac{d}{dt}\mathbb E[K_i(t)]
=
\beta_i(t),
\end{equation}
and
\begin{equation}
\label{eq:moment-S}
\frac{d}{dt}\mathbb E[S_i(t)]
=
(2p-1)
\mathbb E\left[
\frac{S_i(t)}{K_i(t)}\lambda_i(t)
\right].
\end{equation}
Here
\[
\lambda_i(t)
=
\lambda(V_i(t),S_i(t),K_i(t)).
\]
\end{proposition}

\begin{proof}
Because the total event rate affecting a representative neuron is at most
\(\gamma+3\), the number \(J_i(T)\) of events affecting neuron \(i\) before time
\(T\) is stochastically dominated by a Poisson random variable with parameter
\((\gamma+3)T\). Along every sample path,
\[
V_i(t)\leq V_i(0)+J_i(t),
\]
because only incoming spikes increase \(V_i\), each by one, while leakage and
firing reset it to zero. Similarly,
\[
K_i(t)\leq K_i(0)+J_i(t),
\]
and, since every reinforcement update changes \(S_i\) by exactly one,
\[
|S_i(t)|\leq |S_i(0)|+J_i(t).
\]
Taking expectations proves the uniform finite-time first-moment bound.

For \(R\geq1\), let
\[
\chi_R(x):=(-R)\vee(x\wedge R)
\]
and define the bounded test functions
\[
f_R^V(v,s,k)=v\wedge R,
\qquad
f_R^K(v,s,k)=k\wedge R,
\qquad
f_R^S(v,s,k)=\chi_R(s).
\]
The weak equation applies to all three functions. The jump increments of these
truncations are bounded by the corresponding increments of \(v\), \(k\), and
\(|s|\). The finite-time moment estimate therefore provides an integrable
dominating random variable. Consequently, dominated convergence permits
\(R\to\infty\) in the weak equations.

For \(f(v,s,k)=v\), the three parts of the generator are
\[
\gamma\mathbf 1_{\{v>0\}}(0-v)=-\gamma v,
\]
\[
\lambda(v,s,k)\sum_{\xi}q^\xi(s,k)(0-v)
=-v\lambda(v,s,k),
\]
and
\[
a_i(t)((v+1)-v)=a_i(t).
\]
Thus
\[
\mathcal A_{i,\boldsymbol\mu(t)}v
=
-\gamma v-v\lambda(v,s,k)+a_i(t),
\]
which yields \eqref{eq:moment-V} after integration.

For \(f(v,s,k)=k\), leakage and incoming spikes do not change \(k\), while a
firing increases it by one. Hence
\[
\mathcal A_{i,\boldsymbol\mu(t)}k
=
\lambda(v,s,k),
\]
and \eqref{eq:moment-K} follows.

For \(f(v,s,k)=s\), only firing contributes. Since
\[
\sum_{\xi\in\{-1,+1\}}q^\xi(s,k)\xi
=q^+(s,k)-q^-(s,k)
=(2p-1)\frac{s}{k},
\]
we have
\[
\mathcal A_{i,\boldsymbol\mu(t)}s
=
(2p-1)\lambda(v,s,k)\frac{s}{k}.
\]
The ratio is uniformly bounded by one in absolute value because \(|s|\leq k\).
This proves \eqref{eq:moment-S}.
\end{proof}

For the hard-threshold rate, \eqref{eq:moment-V} becomes
\[
\boxed{
\frac{d}{dt}\mathbb E[V_i(t)]
=
\beta_{i-1}(t)+\beta_{i+1}(t)
-
\gamma\mathbb E[V_i(t)]
-
\mathbb E\left[
V_i(t)\mathbf 1_{\{V_i(t)>\alpha S_i^+(t)\}}
\right].
}
\]
The equations are exact but are not closed, because they involve mixed moments
of \(V_i\), \(S_i\), and \(K_i\).

\subsection{Spatially homogeneous RMF}

The ring structure is invariant under cyclic permutations. Let
\[
\vartheta(i)=i+1\pmod N
\]
be the one-step cyclic shift.

\begin{proposition}
\label{prop:spatial-homogeneity}
Assume that the initial laws are translation invariant:
\[
\mu_{i,0}=\mu_0,
\qquad i=1,\ldots,N.
\]
Then the unique nonlinear RMF solution satisfies
\[
\mu_i(t)=\mu(t),
\qquad
\beta_i(t)=\beta(t),
\qquad i=1,\ldots,N,
\]
for every \(t\geq0\). Consequently, every representative neuron receives total
input intensity
\[
a(t)=2\beta(t).
\]
\end{proposition}

\begin{proof}
Let
\[
\boldsymbol\mu(t)
=(\mu_1(t),\ldots,\mu_N(t))
\]
be the unique solution furnished by
Theorem~\ref{thm:existence-uniqueness-rmf}. Define the shifted family
\[
\widetilde\mu_i(t):=\mu_{\vartheta(i)}(t).
\]
Its firing rates satisfy
\[
\widetilde\beta_i(t)
=
\int_E\lambda(x)\,\widetilde\mu_i(t)(dx)
=
\beta_{\vartheta(i)}(t).
\]
Because the network is a ring, the neighbours of \(\vartheta(i)\) are precisely
\(\vartheta(i-1)\) and \(\vartheta(i+1)\). Hence
\begin{align*}
\widetilde\beta_{i-1}(t)+\widetilde\beta_{i+1}(t)
=
\beta_{\vartheta(i-1)}(t)+\beta_{\vartheta(i+1)}(t)
=
\beta_{\vartheta(i)-1}(t)+\beta_{\vartheta(i)+1}(t).
\end{align*}
Therefore the shifted family satisfies exactly the same nonlinear weak equations
as the original family, with shifted indices.

At time zero,
\[
\widetilde\mu_i(0)
=
\mu_{\vartheta(i),0}
=
\mu_0
=
\mu_{i,0}.
\]
Thus \(\boldsymbol\mu\) and \(\widetilde{\boldsymbol\mu}\) solve the same nonlinear
problem with the same initial law. Uniqueness implies
\[
\widetilde\mu_i(t)=\mu_i(t),
\]
so
\[
\mu_{i+1}(t)=\mu_i(t)
\]
for every \(i\). Iterating around the ring gives a common law \(\mu(t)\) for all
types. Integrating \(\lambda\) against this common law gives a common firing rate
\(\beta(t)\), and therefore
\[
a_i(t)=\beta_{i-1}(t)+\beta_{i+1}(t)=2\beta(t).
\]
\end{proof}

Under the hypotheses of Proposition~\ref{prop:spatial-homogeneity}, the nonlinear
generator reduces to
\begin{align}
\mathcal A_{\mu(t)}f(v,s,k)
={}&
\gamma\mathbf 1_{\{v>0\}}
\bigl[f(0,s,k)-f(v,s,k)\bigr]+
2\beta(t)
\bigl[f(v+1,s,k)-f(v,s,k)\bigr]
\label{eq:homogeneous-generator}
\\
&+\nonumber
\lambda(v,s,k)
\sum_{\xi\in\{-1,+1\}}
q^\xi(s,k)
\bigl[f(0,s+\xi,k+1)-f(v,s,k)\bigr],
\end{align}
where
\[
\boxed{
\beta(t)
=
\int_E\lambda(x)\,\mu(t)(dx)
=
\mathbb P\bigl(V(t)>\alpha S^+(t)\bigr).
}
\]
The homogeneous forward equations are obtained from
Proposition~\ref{prop:forward-equations} by replacing \(a_i(t)\) with
\(2\beta(t)\). In particular, for \(v\geq1\),
\begin{align*}
\frac{d}{dt}p(t;v,s,k)
=
2\beta(t)p(t;v-1,s,k)
-
\bigl[2\beta(t)+\gamma+\lambda(v,s,k)\bigr]
p(t;v,s,k),
\end{align*}
and the equation at \(v=0\) is
\begin{align*}
\frac{d}{dt}p(t;0,s,k)
={}&
-2\beta(t)p(t;0,s,k)
+
\gamma\sum_{u\geq1}p(t;u,s,k)
\\
&+
q^+(s-1,k-1)
\sum_{u\geq0}\lambda(u,s-1,k-1)p(t;u,s-1,k-1)
\\
&+
q^-(s+1,k-1)
\sum_{u\geq0}\lambda(u,s+1,k-1)p(t;u,s+1,k-1).
\end{align*}
If the initial first moments are finite, the homogeneous mean-potential equation
is
\[
\boxed{
\frac{d}{dt}\mathbb E[V(t)]
=
2\beta(t)
-
\gamma\mathbb E[V(t)]
-
\mathbb E\left[
V(t)\mathbf 1_{\{V(t)>\alpha S^+(t)\}}
\right].
}
\]

\subsection{Non-explosion}

\begin{proposition}
Assume
$
\lambda(v,s,k)
=
\mathbf 1_{\{v>\alpha s^{+}\}}.
$
For every deterministic locally integrable family
$
(\beta_i(t))_{1\leq i\leq N}
$
satisfying
\[
0\leq\beta_i(t)\leq1,
\]
the RMF limiting process is non-explosive.
\end{proposition}

\begin{proof}
For a representative neuron of type \(i\), the total event rate at time \(t\)
is bounded by
\begin{align*}
\Lambda_i(t)
&=
\gamma\mathbf 1_{\{V_i(t)>0\}}
+
\lambda(V_i(t),S_i(t),K_i(t))
+
\beta_{i-1}(t)
+
\beta_{i+1}(t)
\\
&\leq
\gamma+1+1+1
=
\gamma+3.
\end{align*}
Therefore, the number of events affecting neuron \(i\) up to time \(T\) is
stochastically dominated by a Poisson random variable with parameter
$
(\gamma+3)T.
$
For the vector of \(N\) representative neuron types, the total rate is bounded
by
$
N(\gamma+3).
$
Consequently, only finitely many events occur on every finite time interval,
almost surely.
\end{proof}

\subsection{Existence and uniqueness}

We now establish that the nonlinear process introduced above is well defined.
Throughout this section, we assume that \(N\geq 3\) and that the neuron types
are arranged on the discrete ring
$
I=\{1,\ldots,N\}.
$
All indices are understood modulo \(N\).
Let
\[
E
=
\left\{
(v,s,k)\in
\mathbb N\times\mathbb Z\times\mathbb N^{*}:
|s|\leq k
\right\}.
\]
For \(x=(v,s,k)\in E\), define
\[
\lambda(x)
=
\lambda(v,s,k)
=
\mathbf 1_{\{v>\alpha s^{+}\}},
\qquad
s^{+}=\max\{s,0\}.
\]
The reinforcement probabilities are
\[
q^{+}(s,k)
=
\frac12+
\frac{2p-1}{2}\frac{s}{k},
\]
and
\[
q^{-}(s,k)
=
\frac12-
\frac{2p-1}{2}\frac{s}{k}.
\]
Since \(|s|\leq k\), both quantities belong to \([0,1]\).

For a family of probability measures
\[
\boldsymbol{\mu}(t)
=
\bigl(\mu_1(t),\ldots,\mu_N(t)\bigr),
\]
define
\[
\beta_i(t)
=
\int_E \lambda(x)\,\mu_i(t)(dx).
\]
Since \(0\leq\lambda\leq1\), we have
\[
0\leq\beta_i(t)\leq1.
\]

The total mean-field input rate to a neuron of type \(i\) is
\[
a_i(t)
=
\beta_{i-1}(t)+\beta_{i+1}(t).
\]

For a bounded function \(f:E\to\mathbb R\), define
\begin{align}
\mathcal A_{i,\boldsymbol{\mu}(t)}f(v,s,k)
={}&
\gamma\mathbf 1_{\{v>0\}}
\bigl[
f(0,s,k)-f(v,s,k)
\bigr]+
a_i(t)
\bigl[
f(v+1,s,k)-f(v,s,k)
\bigr]
\label{eq:nonlinear-generator-existence}
\\
&+
\lambda(v,s,k)
\sum_{\xi\in\{-1,+1\}}
q^\xi(s,k)
\bigl[
f(0,s+\xi,k+1)-f(v,s,k)
\bigr].
\nonumber
\end{align}

\begin{theorem}
\label{thm:existence-uniqueness-rmf}
Let
$
\boldsymbol{\mu}_0
=
(\mu_{1,0},\ldots,\mu_{N,0})
\in\mathcal P(E)^N
$
be an arbitrary family of initial probability measures.

Then there exists a unique family of laws
$
\boldsymbol{\mu}(t)
=
(\mu_1(t),\ldots,\mu_N(t))$, 
$t\geq0$,
such that, for every bounded function \(f:E\to\mathbb R\),
\[
\int_E f(x)\,\mu_i(t)(dx)
=
\int_E f(x)\,\mu_{i,0}(dx)
+
\int_0^t
\int_E
\mathcal A_{i,\boldsymbol{\mu}(u)}f(x)
\,\mu_i(u)(dx)\,du.
\]

Equivalently, there exists a unique in law family of nonlinear jump processes
\[
X_i(t)
=
\bigl(V_i(t),S_i(t),K_i(t)\bigr),
\qquad i=1,\ldots,N,
\]
such that
\[
\mathcal L(X_i(t))=\mu_i(t)
\]
and
\[
\beta_i(t)
=
\mathbb E\left[
\mathbf 1_{\{
V_i(t)>\alpha S_i^{+}(t)
\}}
\right].
\]

Moreover, each process is non-explosive.
\end{theorem}

\begin{proof}
The proof is divided into four steps.

\medskip

\noindent
\textbf{Step 1: Construction for prescribed input intensities.}

Fix \(T>0\), and let
$
\boldsymbol{\beta}
=
(\beta_1,\ldots,\beta_N)
$
be a family of measurable functions on \([0,T]\) satisfying
$
0\leq\beta_i(t)\leq1.
$
Define
\[
a_i^{\boldsymbol{\beta}}(t)
=
\beta_{i-1}(t)+\beta_{i+1}(t).
\]
Then
\[
0\leq a_i^{\boldsymbol{\beta}}(t)\leq2.
\]

For every neuron type \(i\), consider a time-inhomogeneous Markov jump process
\[
X_i^{\boldsymbol{\beta}}(t)
=
\left(
V_i^{\boldsymbol{\beta}}(t),
S_i^{\boldsymbol{\beta}}(t),
K_i^{\boldsymbol{\beta}}(t)
\right)
\]
with initial law \(\mu_{i,0}\) and the following transitions:
$
(v,s,k)\longmapsto(0,s,k)
$
at rate
$
\gamma\mathbf 1_{\{v>0\}},
$
$
(v,s,k)\longmapsto(0,s+\xi,k+1)
$
at rate
$
\lambda(v,s,k)q^\xi(s,k)$,
$\xi\in\{-1,+1\},
$ 
and
$
(v,s,k)\longmapsto(v+1,s,k)
$
at rate
$
a_i^{\boldsymbol{\beta}}(t)$.

The total transition rate is bounded by
\begin{align*}
\Lambda_i^{\boldsymbol{\beta}}(t,v,s,k)
&=
\gamma\mathbf 1_{\{v>0\}}
+
\lambda(v,s,k)
+
a_i^{\boldsymbol{\beta}}(t)
\leq
\gamma+1+2
=
\gamma+3.
\end{align*}
Therefore the number of jumps of
\(X_i^{\boldsymbol{\beta}}\) on every finite interval is stochastically
dominated by a Poisson random variable with finite parameter.

Consequently, for every prescribed
\(\boldsymbol{\beta}\), there exists a unique non-explosive
time-inhomogeneous Markov process
$
X_i^{\boldsymbol{\beta}}.
$ Let
\[
\mu_i^{\boldsymbol{\beta}}(t)
=
\mathcal L\left(X_i^{\boldsymbol{\beta}}(t)\right).
\]

\medskip

\noindent
\textbf{Step 2: Definition of the self-consistency map.}

Let
\[
\mathcal B_T
=
\left\{
\boldsymbol{\beta}:[0,T]\to[0,1]^N:
\boldsymbol{\beta}\text{ is measurable}
\right\}.
\]
Equip this space with the norm
$
\|\boldsymbol{\beta}\|_T
=
\max_{1\leq i\leq N}
\operatorname*{ess\,sup}_{0\leq t\leq T}
|\beta_i(t)|.
$

Define
$
\Phi:\mathcal B_T\longrightarrow\mathcal B_T
$
by
\[
\Phi_i(\boldsymbol{\beta})(t)
=
\mathbb E\left[
\lambda\left(
X_i^{\boldsymbol{\beta}}(t)
\right)
\right].
\]
More explicitly,
\[
\Phi_i(\boldsymbol{\beta})(t)
=
\mathbb P\left(
V_i^{\boldsymbol{\beta}}(t)
>
\alpha
\left(
S_i^{\boldsymbol{\beta}}(t)
\right)^{+}
\right).
\]
Since \(0\leq\lambda\leq1\), the map \(\Phi\) indeed takes values in
\(\mathcal B_T\).

A fixed point
$
\boldsymbol{\beta}=\Phi(\boldsymbol{\beta})
$
is exactly a self-consistent family of RMF firing intensities.

\medskip

\noindent
\textbf{Step 3: Contraction estimate on a short time interval.}

Let
$
\boldsymbol{\beta},
\widehat{\boldsymbol{\beta}}
\in\mathcal B_T.
$
We couple
$
X_i^{\boldsymbol{\beta}}
$ and $
X_i^{\widehat{\boldsymbol{\beta}}}
$
using the same initial condition and the same randomness for leakage, firing,
and reinforcement.

For the input processes, use the standard common-part decomposition. At time
\(t\), both processes receive a common input at rate
$
\min\left\{
a_i^{\boldsymbol{\beta}}(t),
a_i^{\widehat{\boldsymbol{\beta}}}(t)
\right\},
$
only \(X_i^{\boldsymbol{\beta}}\) receives an input at rate
$
\left(
a_i^{\boldsymbol{\beta}}(t)
-
a_i^{\widehat{\boldsymbol{\beta}}}(t)
\right)^{+},
$
and only \(X_i^{\widehat{\boldsymbol{\beta}}}\) receives an input at rate
$
\left(
a_i^{\widehat{\boldsymbol{\beta}}}(t)
-
a_i^{\boldsymbol{\beta}}(t)
\right)^{+}.
$

As long as no unmatched input has occurred, the two processes remain equal.
Indeed, they start from the same state and use the same leakage clocks, firing
clocks, and reinforcement marks.

Let
\[
\tau_i
=
\inf\left\{
t\geq0:
\text{an unmatched input occurs for neuron type }i
\right\}.
\]
Then
\[
X_i^{\boldsymbol{\beta}}(t)
=
X_i^{\widehat{\boldsymbol{\beta}}}(t)
\qquad
\text{on the event }\{\tau_i>t\}.
\]
Therefore
\begin{align*}
\left|
\Phi_i(\boldsymbol{\beta})(t)
-
\Phi_i(\widehat{\boldsymbol{\beta}})(t)
\right|
=
\left|
\mathbb E\left[
\lambda(X_i^{\boldsymbol{\beta}}(t))
-
\lambda(X_i^{\widehat{\boldsymbol{\beta}}}(t))
\right]
\right|
\leq
\mathbb E\left[
\left|
\lambda(X_i^{\boldsymbol{\beta}}(t))
-
\lambda(X_i^{\widehat{\boldsymbol{\beta}}}(t))
\right|
\right].
\end{align*}
Since \(0\leq\lambda\leq1\),
\[
\left|
\lambda(X_i^{\boldsymbol{\beta}}(t))
-
\lambda(X_i^{\widehat{\boldsymbol{\beta}}}(t))
\right|
\leq
\mathbf 1_{\{
X_i^{\boldsymbol{\beta}}(t)
\neq
X_i^{\widehat{\boldsymbol{\beta}}}(t)
\}}.
\]
Hence
\[
\left|
\Phi_i(\boldsymbol{\beta})(t)
-
\Phi_i(\widehat{\boldsymbol{\beta}})(t)
\right|
\leq
\mathbb P(\tau_i\leq t).
\]

The unmatched input process has instantaneous rate
$
\left|
a_i^{\boldsymbol{\beta}}(u)
-
a_i^{\widehat{\boldsymbol{\beta}}}(u)
\right|.
$
Thus
\begin{align*}
\mathbb P(\tau_i\leq t)
\leq
\int_0^t
\left|
a_i^{\boldsymbol{\beta}}(u)
-
a_i^{\widehat{\boldsymbol{\beta}}}(u)
\right|
du\leq
\int_0^t
\left(
|\beta_{i-1}(u)-\widehat\beta_{i-1}(u)|
+
|\beta_{i+1}(u)-\widehat\beta_{i+1}(u)|
\right)
du.
\end{align*}
Consequently,
\[
\left|
\Phi_i(\boldsymbol{\beta})(t)
-
\Phi_i(\widehat{\boldsymbol{\beta}})(t)
\right|
\leq
2t
\|\boldsymbol{\beta}
-
\widehat{\boldsymbol{\beta}}\|_T.
\]
Taking the supremum over \(t\in[0,T]\) and the maximum over \(i\), we obtain
\[
\boxed{
\|\Phi(\boldsymbol{\beta})
-
\Phi(\widehat{\boldsymbol{\beta}})\|_T
\leq
2T
\|\boldsymbol{\beta}
-
\widehat{\boldsymbol{\beta}}\|_T.
}
\]

If
$
T<\frac12,
$
then \(2T<1\), and \(\Phi\) is a contraction on \(\mathcal B_T\).
Since \(\mathcal B_T\), modulo equality almost everywhere, is a closed subset
of the Banach space
$
L^\infty([0,T];\mathbb R^N),
$
the Banach fixed-point theorem gives a unique fixed point
$
\boldsymbol{\beta}^{*}\in\mathcal B_T.
$
The processes
$
X_i^{\boldsymbol{\beta}^{*}}
$
are therefore the unique nonlinear RMF processes on \([0,T]\).

\medskip

\noindent
\textbf{Step 4: Extension to arbitrary times.}

Choose \(h\in(0,1/2)\). The previous argument gives a unique solution on
$
[0,h].
$
The laws at time \(h\),
\[
\mu_i(h)
=
\mathcal L(X_i(h)),
\]
can then be used as initial laws for the same fixed-point problem on
$
[h,2h].
$
The contraction estimate is unchanged because it depends only on the length
of the interval and not on the initial distributions.

Repeating the construction gives a unique solution on
$
[0,nh]
$
for every integer \(n\geq1\). Since every finite time belongs to some interval
\([0,nh]\), this defines a global solution for all \(t\geq0\).
Uniqueness on each short interval implies uniqueness on the whole positive
time axis.

Finally, because the total rate of each representative neuron is bounded by
$
\gamma+3,
$
the resulting nonlinear process is non-explosive.
This completes the proof.
\end{proof}

\section{Invariant measures and absence of active stationarity}

The reinforcement counter \(K\) is increased by one at every firing and is
never decreased. Consequently, the full process
\[
X(t)=(V(t),S(t),K(t))
\]
does not admit a nontrivial stationary regime with positive firing activity.

We prove that every invariant probability measure is necessarily concentrated
on silent configurations.

\begin{theorem}
\label{thm:invariant-rmf}
Consider the nonlinear replica mean-field process with firing rate
$
\lambda(v,s,k)
=
\mathbf 1_{\{v>\alpha s^{+}\}},
$
leakage rate
$
\gamma\mathbf 1_{\{v>0\}}$,
 $\gamma>0,
$
and mean-field input rates
\[
a_i(\boldsymbol{\pi})
=
\beta_{i-1}(\pi_{i-1})
+
\beta_{i+1}(\pi_{i+1}),
\]
where
\[
\beta_i(\pi_i)
=
\int_E \lambda(v,s,k)\,\pi_i(dv,ds,dk).
\]

Let
$
\boldsymbol{\pi}
=
(\pi_1,\ldots,\pi_N)
$
be an invariant family of probability measures for the nonlinear RMF
dynamics. Then, for every \(i\),
\[
\beta_i(\pi_i)=0
\]
and
\[
\pi_i\bigl(\{(v,s,k)\in E:v=0\}\bigr)=1.
\]

Conversely, every family of probability measures
$
\boldsymbol{\pi}
=
(\pi_1,\ldots,\pi_N)
$
satisfying
\[
\pi_i(V=0)=1,
\qquad i=1,\ldots,N,
\]
is invariant. 
Thus the invariant probability measures are exactly the probability measures
supported on silent states:
\[
\boxed{
\operatorname{supp}(\pi_i)
\subseteq
\{0\}\times\mathbb Z\times\mathbb N^{*}.
}
\]
\end{theorem}

\begin{proof}
Suppose that
$
\boldsymbol{\pi}
=
(\pi_1,\ldots,\pi_N)
$
is invariant. 
For \(n\geq1\), define the bounded test function
\[
f_n(v,s,k)=k\wedge n.
\]
Leakage and incoming spikes do not change \(k\). At a firing event,
$
k\longmapsto k+1.
$
Therefore,
\begin{align*}
\mathcal A_{i,\boldsymbol{\pi}}f_n(v,s,k)
=
\lambda(v,s,k)
\left[
(k+1)\wedge n-k\wedge n
\right]
=
\lambda(v,s,k)\mathbf 1_{\{k<n\}}.
\end{align*}

Since \(\pi_i\) is invariant and \(f_n\) is bounded,
\[
\int_E
\mathcal A_{i,\boldsymbol{\pi}}f_n(v,s,k)
\,\pi_i(dv,ds,dk)
=
0.
\]
Hence
\[
\int_E
\lambda(v,s,k)\mathbf 1_{\{k<n\}}
\,\pi_i(dv,ds,dk)
=
0.
\]

The integrand is nonnegative. Therefore,
\[
\lambda(v,s,k)\mathbf 1_{\{k<n\}}=0
\qquad
\pi_i\text{-almost surely}.
\]
Letting \(n\to\infty\), monotone convergence gives
\[
\int_E\lambda(v,s,k)\,\pi_i(dv,ds,dk)=0.
\]
Thus
\[
\boxed{
\beta_i(\pi_i)=0.
}
\]

This holds for every neuron type \(i\). Consequently, all mean-field input
rates vanish:
\[
a_i(\boldsymbol{\pi})
=
\beta_{i-1}(\pi_{i-1})
+
\beta_{i+1}(\pi_{i+1})
=
0.
\]

We next prove that the potential must vanish almost surely.

Consider the bounded test function
\[
g(v,s,k)=\mathbf 1_{\{v=0\}}.
\]
Since
\[
\lambda(v,s,k)=0
\qquad
\pi_i\text{-almost surely},
\]
and since \(a_i(\boldsymbol{\pi})=0\), only the leakage transition contributes
to the generator.

If \(v>0\), leakage sends \(v\) to \(0\), and therefore
\[
g(0,s,k)-g(v,s,k)=1.
\]
If \(v=0\), no leakage occurs. Hence
\[
\mathcal A_{i,\boldsymbol{\pi}}g(v,s,k)
=
\gamma\mathbf 1_{\{v>0\}}.
\]

By invariance,
\[
0
=
\int_E
\mathcal A_{i,\boldsymbol{\pi}}g(v,s,k)
\,\pi_i(dv,ds,dk),
\]
and therefore
\[
0
=
\gamma\pi_i(V>0).
\]
Since \(\gamma>0\),
\[
\pi_i(V>0)=0.
\]
Thus
\[
\boxed{
\pi_i(V=0)=1.
}
\]

This proves that every invariant family is supported on silent
configurations.

Conversely, suppose that
\[
\pi_i(V=0)=1
\]
for every \(i\). Then
\[
\lambda(0,s,k)
=
\mathbf 1_{\{0>\alpha s^{+}\}}
=
0,
\]
so
\[
\beta_i(\pi_i)=0.
\]
Hence all mean-field input rates are zero:
$
a_i(\boldsymbol{\pi})=0.
$

Moreover, at \(v=0\),

\[
\gamma\mathbf 1_{\{v>0\}}=0
\]
and
\[
\lambda(v,s,k)=0.
\]
Thus no leakage, firing, reinforcement, or incoming spike can occur. The
process remains at its initial state for all times. Therefore every
probability measure supported on \(V=0\) is invariant.
\end{proof}

\begin{Corollary}
There is no invariant probability measure \(\boldsymbol{\pi}\) satisfying
\[
\beta_i(\pi_i)>0
\]
for at least one neuron type \(i\).

In particular, the full reinforced RMF process has no stationary regime with
a positive firing rate.
\end{Corollary}

\begin{proof}
The conclusion follows immediately from
Theorem~\ref{thm:invariant-rmf}, which implies
\[
\beta_i(\pi_i)=0
\]
for every invariant family.
\end{proof}

\begin{remark}
The obstruction to a nontrivial invariant probability measure is the
monotonicity of the counter \(K\):
\[
K(t)
=
K(0)+N^{\mathrm{fire}}(t),
\]
where \(N^{\mathrm{fire}}(t)\) is the number of firings up to time \(t\).

Thus, whenever the process fires persistently, \(K(t)\) tends to infinity.
A nontrivial stationary analysis would therefore require one of the following
modifications:

\begin{enumerate}
\item studying the reduced variables
$
(V,R),
$$
R=\frac{S}{K};
$

\item considering a process observed in the reinforcement-event time scale;

\item introducing forgetting or decay in the reinforcement variables;

\item studying quasi-stationary distributions conditioned on non-extinction;

\item considering stationary limits only for the potential marginal \(V\).
\end{enumerate}
\end{remark}

\section{Numerical experiments}

We simulate the finite reinforced network using an exact event-driven
algorithm. At each step, the waiting time to the next event is sampled from an
exponential distribution with parameter equal to the total intensity of all
enabled events, and the event is selected proportionally to its rate. In the
original undriven network, these events are firing and leakage. For the
sustained-input experiments below, each neuron additionally receives an
independent homogeneous Poisson input of rate \(\eta\). At an external input
time,
$
V_i\longmapsto V_i+1.
$
This external input is used only in the numerical adaptation study. All
analytical results in the preceding sections concern the original undriven
model, corresponding to \(\eta=0\).

\subsection{Firing-rate adaptation under sustained input}

We first examine whether Elephant memory generates firing-rate adaptation
under sustained stimulation. We take
\[
N=100,
\qquad
\alpha=0.5,
\qquad
\gamma=0.2,
\qquad
\eta=0.4,
\]
with initial condition
\[
V_i(0)=0,
\qquad
S_i(0)=K_i(0)=1,
\qquad
i=1,\ldots,N.
\]
Each simulation is run up to time \(T_{\max}=300\), and firing rates are
evaluated in time windows of length \(\Delta=5\). For each parameter
configuration, we perform \(50\) independent simulations.

We compare the reinforced system for
\[
p\in\{0.6,0.75,0.9\}
\]
with a no-memory control in which the reinforcement updates are disabled and
the firing threshold is fixed at zero. All systems receive external input at
the same rate. For a time window \((t-\Delta,t]\), the population firing rate
is defined by
\[
r(t)
=
\frac{
\text{number of firings in }(t-\Delta,t]
}{
N\Delta
}.
\]
For the reinforced systems, we also record the mean effective firing threshold
\[
\overline{\theta}(t)
=
\frac{\alpha}{N}
\sum_{i=1}^{N}S_i^+(t).
\]

\begin{figure}
\centering
\includegraphics[width=0.75\textwidth]
{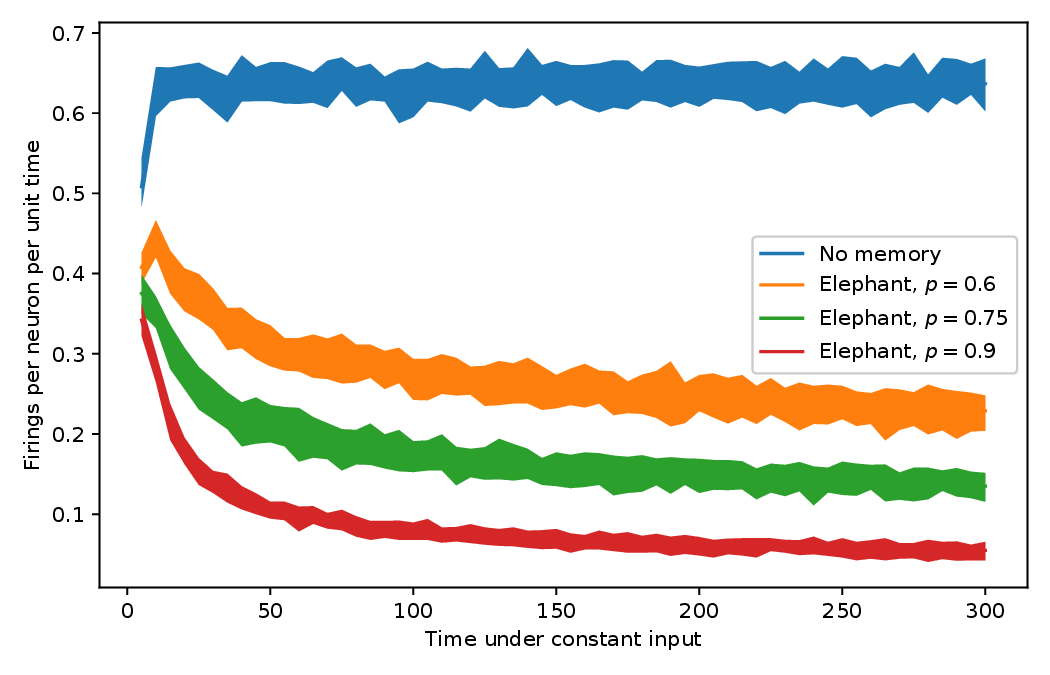}
\caption{Mean population firing rates under sustained homogeneous input. The
no-memory control has a firing threshold fixed at zero, while the reinforced
systems use \(\alpha=0.5\) and the indicated values of \(p\). The shaded
regions represent empirical interquartile ranges over \(50\) realizations.}
\label{fig:driven-adaptation-rates}
\end{figure}

The no-memory control maintains an approximately constant firing rate, whereas
the Elephant systems display a progressive decline in activity. This decline
becomes stronger as \(p\) increases.
Figure~\ref{fig:driven-adaptation-thresholds} shows that the reduction in
firing is accompanied by growth of the effective threshold.

\begin{figure}
\centering
\includegraphics[width=0.75\textwidth]
{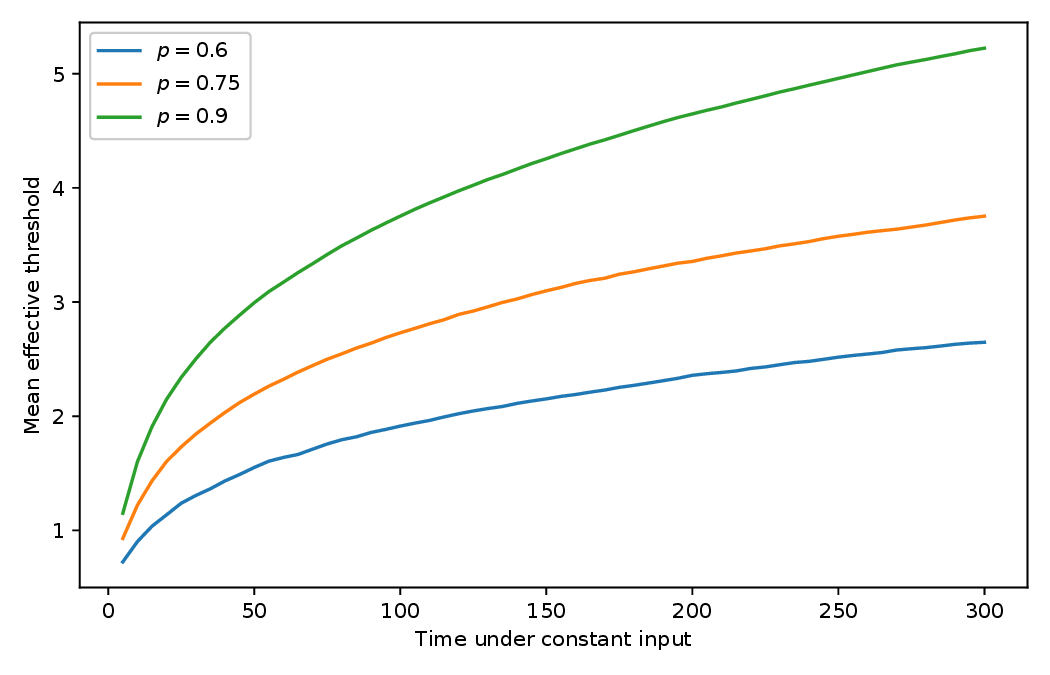}
\caption{Mean effective firing threshold
\(\overline{\theta}(t)\) under sustained homogeneous input for the reinforced
systems. Each curve is averaged over \(50\) realizations.}
\label{fig:driven-adaptation-thresholds}
\end{figure}

To summarize the magnitude of adaptation, for each realization we define the
early firing rate as the average over
\[
10\leq t\leq50
\]
and the late firing rate as the average over
\[
250\leq t\leq300.
\]
The relative decline in that realization is
\[
1-
\frac{\text{late firing rate}}
{\text{early firing rate}}.
\]
Table~\ref{tab:driven-adaptation} reports the averages over the \(50\)
realizations. The early and late thresholds are averaged over the same time
intervals.

\begin{table}
\centering
\begin{tabular}{c|c|c|c|c|c}
\hline
Condition &
Early rate &
Late rate &
Relative decline &
Early threshold &
Late threshold
\\
\hline
No memory & \(0.6354\) & \(0.6369\) & \(-0.3\%\) & \(0\) & \(0\) \\
\(p=0.6\) & \(0.3604\) & \(0.2298\) & \(36.1\%\) & \(1.27\) & \(2.59\) \\
\(p=0.75\) & \(0.2576\) & \(0.1408\) & \(45.4\%\) & \(1.79\) & \(3.66\) \\
\(p=0.9\) & \(0.1604\) & \(0.0560\) & \(65.3\%\) & \(2.42\) & \(5.10\) \\
\hline
\end{tabular}
\caption{Early and late firing rates and mean effective thresholds under
sustained input. Each row is based on \(50\) independent simulations.}
\label{tab:driven-adaptation}
\end{table}

Under the same sustained input, the no-memory system remains approximately
stationary. In contrast, Elephant reinforcement produces a persistent decline
in firing activity accompanied by growth of the effective threshold. Since
the initial reinforcement is positive, larger values of \(p\) favour further
positive reinforcement increments, producing stronger threshold accumulation
and a larger reduction in firing activity.

\subsection{Temporal profile of adaptation}

We next compare the temporal profile of the mean adaptation curves with a
single-exponential decay and a power law. For each value of \(p\), the mean
firing-rate trajectory from time \(t=10\) onward is fitted with
\[
r_{\mathrm{exp}}(t)
=
c+a\exp\left(-\frac{t}{\tau}\right)
\]
and
\[
r_{\mathrm{pow}}(t)
=
c+a(t+1)^{-\beta}.
\]
Both models have three fitted parameters. The fits are compared using the
Akaike information criterion, and their stability is assessed by \(200\)
bootstrap resamples of the \(50\) simulation runs for each value of \(p\).

\begin{figure}
\centering
\includegraphics[width=0.75\textwidth]
{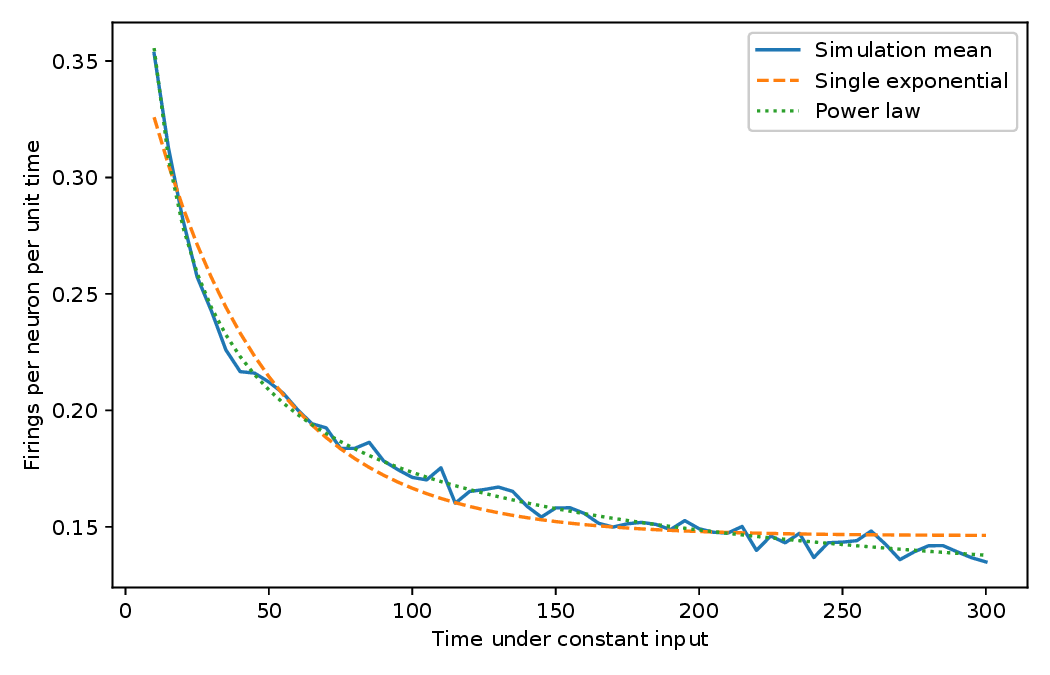}
\caption{Single-exponential and power-law fits to the mean firing-rate
trajectory for \(p=0.75\), using data from \(t=10\) onward.}
\label{fig:adaptation-fit}
\end{figure}

Table~\ref{tab:adaptation-fits} reports the fitted power-law exponent, the
single-exponential time constant, and
\[
\Delta\mathrm{AIC}
=
\mathrm{AIC}_{\mathrm{exp}}
-
\mathrm{AIC}_{\mathrm{pow}}.
\]
Positive values of \(\Delta\mathrm{AIC}\) favour the power-law description.

\begin{table}
\centering
\begin{tabular}{c|c|c|c|c}
\hline
\(p\) &
Power exponent \(\beta\) &
Exponential time \(\tau\) &
\(\Delta\mathrm{AIC}\) &
Bootstrap fraction favouring power law
\\
\hline
\(0.6\)  & \(0.3149\) & \(56.47\) & \(81.72\)  & \(1.00\) \\
\(0.75\) & \(0.5327\) & \(41.39\) & \(104.05\) & \(1.00\) \\
\(0.9\)  & \(0.8154\) & \(27.37\) & \(183.47\) & \(1.00\) \\
\hline
\end{tabular}
\caption{Comparison of power-law and single-exponential descriptions of the
mean adaptation trajectories. The final column is the fraction of \(200\)
bootstrap resamples for which the power-law fit has lower AIC.}
\label{tab:adaptation-fits}
\end{table}

Over the simulated interval, the power-law form provides a substantially
better description of the mean adaptation trajectories than a single
exponential. The power-law fit has lower AIC for every value of \(p\) and in
all \(200\) bootstrap resamples.

\subsection{Effect of the Elephant parameter in the undriven network}

We now return to the original network without external input. Unless otherwise
stated, we take
\[
N=100,
\qquad
V_i(0)=S_i(0)=K_i(0)=1,
\qquad
i=1,\ldots,N.
\]
For each parameter configuration, we perform \(20\) independent simulations
up to time \(T_{\max}=300\). Since the network receives no external input, the
state in which all membrane potentials vanish is absorbing. We define the
extinction time by
\[
\tau_{\mathrm{ext}}
=
\inf
\left\{
t\geq0:
V_i(t)=0
\text{ for every }i
\right\}.
\]

We examine the effect of the memory parameter \(p\), fixing
\[
\alpha=0.5,
\qquad
\gamma=0.2.
\]
Table~\ref{tab:p-results} reports the median extinction time, its empirical
interquartile range, and the mean total number of firings. All simulated
networks became extinct before \(T_{\max}\).

\begin{table}
\centering
\begin{tabular}{c|c|c}
\hline
\(p\) &
Median extinction time
\([Q_{0.25},Q_{0.75}]\) &
Mean total firings
\\
\hline
\(0.1\) & \(135.29\,[108.69,145.95]\) & \(2187.25\) \\
\(0.3\) & \(90.79\,[71.33,111.26]\)   & \(1327.05\) \\
\(0.5\) & \(72.43\,[53.99,85.00]\)    & \(728.85\)  \\
\(0.7\) & \(32.37\,[28.15,41.56]\)    & \(294.55\)  \\
\(0.9\) & \(28.06\,[24.67,33.62]\)    & \(176.15\)  \\
\hline
\end{tabular}
\caption{Extinction times and total firing activity for different values of
the Elephant parameter \(p\). Each row is based on \(20\) independent
simulations.}
\label{tab:p-results}
\end{table}

Both the extinction time and the total number of firings decrease markedly as
\(p\) increases. Since the initial reinforcement is positive, larger values of
\(p\) favour further positive increments of the reinforcement variable. This
raises the firing threshold and progressively suppresses network activity.
For \(p<1/2\), increments opposing the current reinforcement sign are favoured,
which weakens this threshold accumulation and prolongs the active transient.

\subsection{Dependence on the threshold and leakage parameters}

We next vary
\[
\alpha\in\{0.25,0.5,0.75,1\},
\qquad
\gamma\in\{0.1,0.2,0.4,0.8\},
\]
while fixing \(p=0.7\). Figure~\ref{fig:alpha-gamma} displays the median total
number of firings over \(20\) independent simulations for each parameter pair.

\begin{figure}
\centering
\includegraphics[width=0.72\textwidth]
{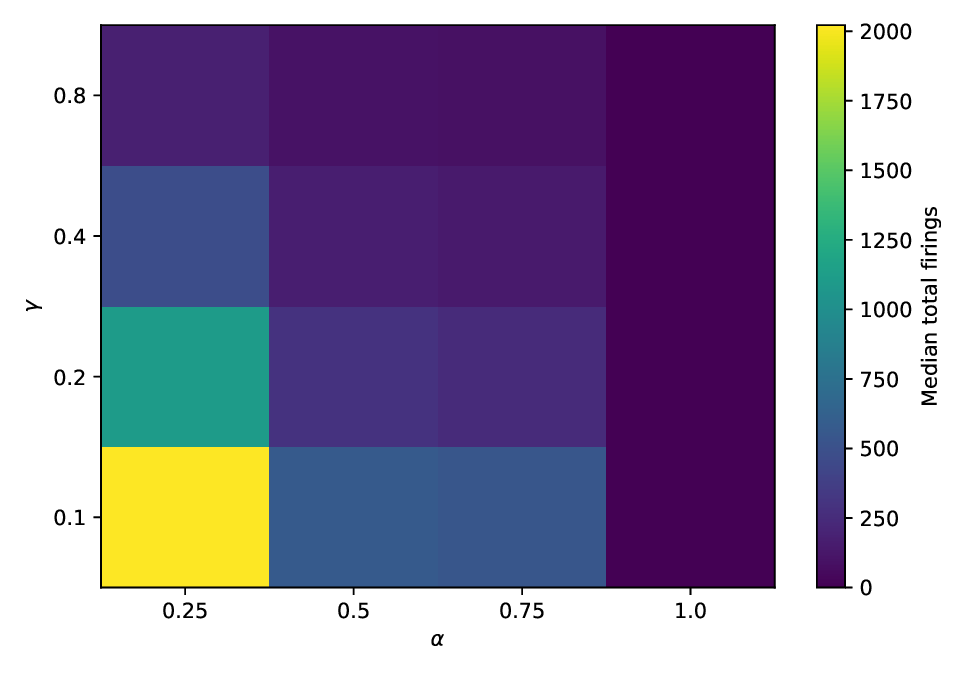}
\caption{Median total number of firings as a function of the threshold
coefficient \(\alpha\) and the leakage rate \(\gamma\), with \(p=0.7\) and
\(N=100\).}
\label{fig:alpha-gamma}
\end{figure}

Increasing either parameter reduces the total firing activity. For
\(\alpha=0.25\), the median number of firings decreases from \(2022\) at
\(\gamma=0.1\) to \(180\) at \(\gamma=0.8\). At \(\gamma=0.2\), increasing
\(\alpha\) from \(0.25\) to \(0.75\) reduces the median from \(1110\) to
\(250\). When \(\alpha=1\), the chosen initial condition satisfies
\[
V_i(0)=\alpha S_i^+(0),
\]
and the strict firing condition
\[
V_i(0)>\alpha S_i^+(0)
\]
fails. Consequently, the network produces no firing events in this case.

\subsection{Finite-network and RMF trajectories}

Finally, we compare the undriven finite network with a particle approximation
of the spatially homogeneous nonlinear RMF process. The empirical RMF firing
intensity is
\[
\widehat\beta^Q(t)
=
\frac{1}{Q}
\sum_{q=1}^{Q}
\mathbf 1_{\{
V^q(t)>\alpha(S^q(t))^+
\}},
\]
and each particle receives inputs at rate \(2\widehat\beta^Q(t)\). We use
\[
Q=1000,
\qquad
p=0.7,
\qquad
\alpha=0.5,
\qquad
\gamma=0.2,
\]
and average \(20\) realizations of each system. Firing rates are evaluated in
time windows of length \(\Delta=2\).

\begin{figure}
\centering
\includegraphics[width=0.75\textwidth]
{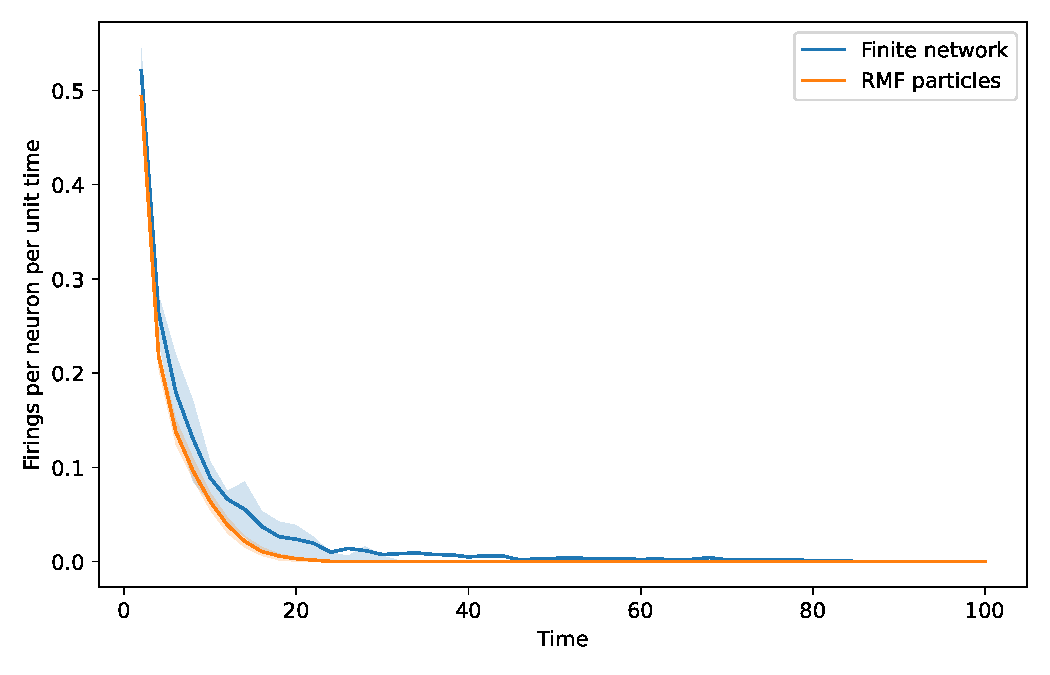}
\caption{Mean population firing rates for the undriven finite network and the
RMF particle approximation. The shaded regions represent empirical
interquartile ranges over \(20\) realizations.}
\label{fig:finite-rmf}
\end{figure}

The two systems exhibit the same qualitative transient: an initially high
firing rate followed by rapid decay toward the silent state. The peak mean
rates are \(0.5205\) for the finite network and \(0.49345\) for the RMF
approximation. Over the interval \([0,100]\), the discrete approximation of
the integrated absolute difference between the two mean firing-rate curves is
\[
0.9206.
\]
The finite system retains slightly more activity during part of the transient,
but the RMF approximation reproduces the overall decay profile.

All simulations were implemented in Python 3.13.5 using NumPy 2.3.5,
SciPy 1.17.0, Matplotlib 3.10.8, and Numba 0.65.1. Explicit fixed random seeds
are specified in the simulation scripts. The scripts, raw simulation outputs,
parameter summaries, and figure-generation files are provided with the
supplementary numerical material.

\section{Discussion}

Taken together, these results show that the model can represent biologically
relevant forms of activity-dependent response modulation. Under sustained
homogeneous input, the no-memory control maintains an approximately stationary
firing rate, whereas the reinforced system exhibits a progressive decline in
activity. Because this decline is accompanied by an accumulated increase in
the effective firing threshold, the model reproduces both spike-frequency
adaptation and activity-dependent threshold adaptation. The dependence on
\(p\) shows that the strength of these effects is controlled by how
persistently previous spikes modify excitability. The model is therefore
particularly relevant for neuronal systems in which repeated firing produces
a lasting reduction in responsiveness that cannot be explained by the
instantaneous membrane potential alone.

The simulations show that this memory can strongly regulate both the duration
and the intensity of network activity. With positive initial reinforcement,
larger values of \(p\) favour the persistence of positive reinforcement,
causing the firing threshold to increase more rapidly. The resulting effect is
a cumulative negative feedback: firing reduces the probability of future
firing. Under sustained input, larger values of \(p\) produce a stronger
decline in firing activity and a greater increase in the effective threshold.
In the undriven network, they also produce shorter extinction times and
substantially fewer spikes, whereas \(p<1/2\) favours reversals of the current
reinforcement tendency and prolongs the active transient.

Over the simulated time interval, the adaptation trajectories are better
described by a power law than by a single-exponential relaxation. This
indicates that the cumulative reinforcement rule can generate an extended
temporal response without introducing several adaptation variables with
separately prescribed decay constants. The comparison is restricted to the
finite time interval and parameter values considered here and does not by
itself establish an asymptotic power law. Moreover, the reinforcement variable
should be interpreted as a parsimonious representation of
spike-history-dependent excitability rather than as a model of a particular
ionic or molecular mechanism.

The parameters \(\alpha\) and \(\gamma\) control complementary mechanisms.
The coefficient \(\alpha\) determines how strongly firing history affects
excitability, while \(\gamma\) determines how rapidly membrane potential is
lost. Increasing either parameter suppresses activity in the undriven
network, showing that recurrent firing may be limited by strong
history-dependent adaptation, rapid membrane relaxation, or both. The
finite-network and RMF simulations also display closely related transient
firing-rate profiles under matched parameter settings, supporting the use of
the nonlinear RMF process as a qualitative approximation of the finite
network dynamics.

The characterization of invariant measures concerns the original undriven
model and shows that cumulative reinforcement ultimately selects silent
stationary configurations. The sustained-input simulations instead examine
adaptation during externally maintained activity and therefore do not
contradict the absence of active invariant measures in the undriven system.
Introducing recovery or decay of the reinforcement state would be a natural
extension for studying post-stimulus recovery, repeated stimulation, and
possible active stationary regimes.

\end{document}